\documentclass[
11pt,
a4paper, 
twoside, 
reqno
]{amsart}

\usepackage[T1]{fontenc} 
\usepackage[utf8]{inputenc} 
\usepackage{graphicx} 
\graphicspath{{Figures/}} 
\usepackage{subfig} 
\usepackage{amsmath,amssymb,amsthm,stmaryrd,mathtools} 
\numberwithin{equation}{section}
\usepackage{varioref} 
\usepackage[shortlabels]{enumitem} 
\usepackage{tikz}
\usetikzlibrary{calc}
\usetikzlibrary{arrows}
\usepackage{shuffle}
\usepackage{comment}
\usepackage{hyperref}

\usetikzlibrary{matrix}
\usetikzlibrary{decorations.markings}

\theoremstyle{plain} 
\newtheorem{theorem}{Theorem}[section]
\newtheorem*{thma}{Theorem A}
\newtheorem*{thmb}{Theorem B}

\newtheorem{proposition}[theorem]{Proposition}
\newtheorem{corollary}[theorem]{Corollary}
\newtheorem{lemma}[theorem]{Lemma}
\newtheorem*{lemma*}{Lemma}

\newtheorem*{question*}{Question}

\theoremstyle{definition} 
\newtheorem{definition}[theorem]{Definition}
\newtheorem{example}[theorem]{Example}

\theoremstyle{remark} 
\newtheorem{remark}[theorem]{Remark}

\newtheorem*{ack}{Acknowledgements}

\newcommand{\Z}{\mathbb Z}
\newcommand{\N}{\mathbb N}
\newcommand{\R}{\mathbb R}
\newcommand{\ASb}{\operatorname{AS}\nolimits}
\newcommand{\APb}{\operatorname{AP}\nolimits}

\newcommand{\DD}{\mathsf D}

\newcommand{\CC}{\mathcal C}
\newcommand{\Se}{\mathbb{S}}
\newcommand{\Le}{\mathcal L}
\newcommand{\Ri}{\mathcal R}
\newcommand{\Jac}{\mathcal P}
\newcommand{\bo}{\operatorname{b}\nolimits}
\newcommand{\Hom}{\operatorname{Hom}\nolimits}

\newcommand{\Ext}{\operatorname{Ext}\nolimits}

\newcommand{\Sub}{\operatorname{Sub}\nolimits}
\newcommand{\projdim}{\mathop{{\rm proj.dim}}\nolimits}
\newcommand{\gldim}{\mathop{{\rm gl.dim\,}}\nolimits}
\newcommand{\otimesk}{\otimes_k}

\newcommand{\op}{\operatorname{op}\nolimits}
\newcommand{\bd}{\DD^{\bo}}

\newcommand{\modu}{\mathsf{mod}\,}
\newcommand{\proj}{\mathsf{proj}\,}
\newcommand{\inj}{\mathsf{inj}\,}
\newcommand{\RHom}{\mathbf{R}\strut\kern-.2em\operatorname{Hom}\nolimits}
\newcommand{\Lotimes}{\mathop{\stackrel{\mathbf{L}}{\otimes}}\nolimits}
\newcommand{\Ho}{\operatorname{H}\nolimits}

\DeclareMathOperator{\socmod}{soc}

\author{Mads Hustad Sand\o y}
\address{Institutt for matematiske fag, NTNU, 7491 Trondheim, Norway}
\email{mads.sandoy@ntnu.no}

\author{Louis-Philippe Thibault}
\address{Institutt for matematiske fag, NTNU, 7491 Trondheim, Norway}
\email{lp.thibault@mail.utoronto.ca}

\begin{document}

\title{Classification results for $n$-hereditary monomial algebras}



\begin{abstract}
We classify $n$-hereditary monomial algebras in three natural contexts: First, we give a classification of the $n$-hereditary truncated path algebras. We show that they are exactly the $n$-representation-finite Nakayama algebras classified by Vaso. Next, we classify partially the $n$-hereditary quadratic monomial algebras. In the case $n=2$, we prove that there are only two examples, provided that the preprojective algebra is a planar quiver with potential. The first one is a Nakayama algebra and the second one is obtained by mutating $\mathbb A_3\otimes_k \mathbb A_3$, where $\mathbb A_3$ is the Dynkin quiver of type $A$ with bipartite orientation. In the case $n\geq 3$, we show that the only $n$-representation finite algebras are the $n$-representation-finite Nakayama algebras with quadratic relations.
\end{abstract}

\thanks{2020 {\em Mathematics Subject Classification.} 16G20 (Primary), 16E30}
\thanks{{\em Key words and phrases.} $n$-hereditary algebra, $n$-representation-finite algebra, $n$-representation-infinite algebra, preprojective algebra, Jacobian algebra, selfinjective algebra, Calabi--Yau algebra, Auslander--Reiten theory. }
\maketitle
\tableofcontents

\setcounter{tocdepth}{2}


\section{Introduction}

Auslander--Reiten theory has proven to be a central tool in the study of the representation theory of Artin algebras \cite{ARS97}. In 2004, Iyama introduced a generalisation of some of the key concepts to a `higher-dimensional' paradigm \cite{Iya07b, Iya07}. To put it in his own words, ``in these Auslander--Reiten theories, the number `2' is quite symbolic''. For example, the Auslander correspondence establishes a bijection between finite-dimensional representation-finite algebras and finite-dimensional algebras of global dimension at most $2$ and dominant dimension at least $2$ \cite{Aus71}. This realisation was the starting point of very fruitful research which has had applications in representation theory, commutative algebra, as well as commutative and categorical algebraic geometry (e.g. \cite{Iya11, IO11, MM11, IO13, HIO14, HIMO14, IW14, AIR15, IJ17, DJW19, JK19, BH}). 

Auslander--Reiten theory is particularly nice over finite-dimensional hereditary algebras $\Lambda$. For example, there is a trichotomy in the representation theory of these algebras into preprojective, regular and preinjective modules. Moreover, their preprojective algebra $\Pi = T_{\Lambda}\Ext^1_{\Lambda}(D\Lambda, \Lambda)$ provides very useful information \cite{BGL87}. This motivated the study of the so-called $n$-hereditary algebras, which consist of the $n$-representation-finite (henceforth abbreviated as $n$-RF) \cite{Iya07, HI11b, HI11, Iya11, IO11, IO13} and $n$-representation-infinite (henceforth $n$-RI) \cite{HIO14} algebras. These are finite-dimensional algebras of global dimension $n$ which enjoy properties analogous to hereditary algebras in the classical theory. There is also a natural generalisation of the preprojective algebra over these algebras. 

Many instances of $n$-hereditary algebras were discovered over the years (e.g. \cite{HI11, IO13, AIR15, Pet19, Pas20, BH}). For example, algebras of higher type $A$ and type $\tilde A$ are $n$-RF and $n$-RI, respectively \cite{IO11, HIO14}. The defining properties of $n$-hereditary algebras are rather strong, so classes of examples should be expected to be somewhat special. However, it seems that we are still in an early stage, and that many more classes of examples and classification results have yet to be discovered. Such results would allow an even better understanding of the role of these algebras.

The aim of this paper is to study characteristics of certain $n$-hereditary monomial algebras. On many occasions, we use the fact that $n$-hereditary algebras $\Lambda$ enjoy the property that $\Ext^j_{\Lambda^e}(\Lambda, \Lambda^e) = 0$ for all $0<j<n$ \cite{IO13}, which we refer to as the vanishing-of-$\Ext$ condition. Since monomial algebras have a nice bimodule resolution, provided by Bardzell \cite{Bar97}, we have good control over these extension groups. Using that fact and a classification of the $n$-representation-finite Nakayama algebras by Vaso \cite{Vas19}, we obtain the following result for truncated path algebras. 

\begin{thma}[Proposition \ref{prop:rad_no_ext}, Theorem \ref{thm:class_trunc}]
Let $\Lambda = kQ/\mathcal J^{\ell}$ be a truncated path algebra, where $\ell\geq 2$, $Q$ is a finite quiver and $\mathcal J$ is the arrow ideal. Let $\mathbb A_m$ be the linearly oriented Dynkin quiver of type $A$ with $m$ vertices. 
\begin{enumerate}
\item If $Q$ is acyclic and $\Ext^j_{\Lambda^e}(\Lambda, \Lambda^e) = 0$ for all $0<j<\gldim\Lambda$, then $Q = \mathbb A_m$, for some $m$. 
\item The following are equivalent: 
	\begin{enumerate}
\item $\Lambda$ is $n$-hereditary;
\item $\Lambda \cong k\mathbb A_m/\mathcal J^{\ell}$, for some $m$, and $\ell \,|\, m-1$ or $\ell =2$.
\end{enumerate}
In this case, $n = 2\frac{m-1}{\ell}$ and $\Lambda$ is an $n$-representation-finite Nakayama algebra. 
\end{enumerate}
\end{thma}

We note that the vanishing-of-$\Ext$ condition already allows us to reduce the number of cases by quite a bit. 

Next, we move to the study of quadratic monomial algebras. Our main results are given as follows.

\begin{thmb}[Theorem \ref{thm:main_quad_mono}, Corollary \ref{cor:star}, Theorem \ref{thm:mono_higher}] 
Let $\Lambda = kQ/I$ be a quadratic monomial algebra of global dimension $n$. 
\begin{enumerate}
\item Suppose that $n =2$. 
	\begin{enumerate}
		\item If $\Ext^1_{\Lambda^e}(\Lambda, \Lambda^e) = 0$, then $Q$ is an $(r,s)$-star quiver (Definition \ref{def:star}).
		\item If $\Lambda$ is $n$-hereditary and the preprojective algebra $\Pi(\Lambda)$ is a planar QP, then $\Lambda$ is given by one of the following two $2$-RF algebras: 
		
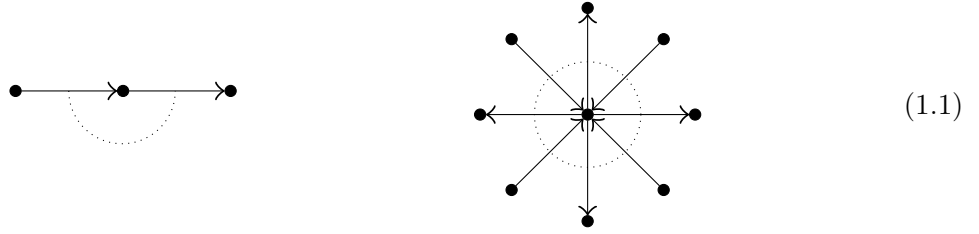
\begin{equation}
\label{two_quivers}
\begin{tikzpicture}[auto, baseline=(current  bounding  box.center)]
	\node[circle, draw, thin,fill=black!100, scale=0.4] (0) at (0,1) {};
	\node[circle, draw, thin,fill=black!100, scale=0.4] (1) at (1.414,1) {};
	\node[circle, draw, thin,fill=black!100, scale=0.4] (2) at (2.828,1) {};
	\draw[decoration={markings,mark=at position 1 with {\arrow[scale=2]{>}}},
    postaction={decorate}, shorten >=0.4pt] (0) to (1);
	\draw[decoration={markings,mark=at position 1 with {\arrow[scale=2]{>}}},
    postaction={decorate}, shorten >=0.4pt] (1) to (2);
	\draw[dotted] (0.7,1) arc (-180:0:.7){};
\end{tikzpicture}\qquad\qquad\qquad\qquad
\begin{tikzpicture}[auto, baseline=(current  bounding  box.center)]
	\node[circle, draw, thin,fill=black!100, scale=0.4] (0) at (0,0) {};
	\node[circle, draw, thin,fill=black!100, scale=0.4] (1) at (1,-0.414) {};
	\node[circle, draw, thin,fill=black!100, scale=0.4] (2) at (2,0) {};
	\node[circle, draw, thin,fill=black!100, scale=0.4] (3) at (-0.414,1) {};
	\node[circle, draw, thin,fill=black!100, scale=0.4] (4) at (1,1) {};
	\node[circle, draw, thin,fill=black!100, scale=0.4] (5) at (2.414,1) {};
	\node[circle, draw, thin,fill=black!100, scale=0.4] (6) at (0,2) {};
	\node[circle, draw, thin,fill=black!100, scale=0.4] (7) at (1,2.414) {};
	\node[circle, draw, thin,fill=black!100, scale=0.4] (8) at (2,2) {};	
	
	\draw[decoration={markings,mark=at position 1 with {\arrow[scale=2]{>}}},
    postaction={decorate}, shorten >=0.4pt] (4) to (1);
	\draw[decoration={markings,mark=at position 1 with {\arrow[scale=2]{>}}},
    postaction={decorate}, shorten >=0.4pt] (4) to (5);
	\draw[decoration={markings,mark=at position 1 with {\arrow[scale=2]{>}}},
    postaction={decorate}, shorten >=0.4pt] (4) to (7);
	\draw[decoration={markings,mark=at position 1 with {\arrow[scale=2]{>}}},
    postaction={decorate}, shorten >=0.4pt] (4) to (3);
	\draw[decoration={markings,mark=at position 1 with {\arrow[scale=2]{>}}},
    postaction={decorate}, shorten >=0.4pt] (0) to (4);
	\draw[decoration={markings,mark=at position 1 with {\arrow[scale=2]{>}}},
    postaction={decorate}, shorten >=0.4pt] (2) to (4);
	\draw[decoration={markings,mark=at position 1 with {\arrow[scale=2]{>}}},
    postaction={decorate}, shorten >=0.4pt] (6) to (4);
	\draw[decoration={markings,mark=at position 1 with {\arrow[scale=2]{>}}},
    postaction={decorate}, shorten >=0.4pt] (8) to (4);
	
	\draw[dotted] (1.7,1) arc (0:360:.7){};
\end{tikzpicture}
\end{equation}
where the dotted arcs denote relations. Note that the first algebra is the Nakayama algebra $k\mathbb A_3/\mathcal J^2$.
	\end{enumerate}
\item Suppose that $n\geq 3$ and $\Lambda$ is $n$-RF. Then $\Lambda \cong k\mathbb A_{n+1}/\mathcal J^{2}$.
\end{enumerate}
\end{thmb}

Perhaps surprisingly, we see that the class of $2$-RF quadratic monomial algebras is richer than those in higher global dimension. In the $n=2$ case, we assumed that the preprojective algebra was a planar quiver with potential. There are examples of other $2$-RF quadratic monomial algebras where this property is not satisfied, see Example \ref{ex:2rf}. This assumption appears often, at least implicitly, in different results aimed at understanding some selfinjective Jacobian algebras and $2$-RF algebras (e.g. \cite{HI11, Pet19, Pas20}). Note that all examples covered in the previous theorem were already known to be $n$-RF. The algebra corresponding to the $(4,4)$-star above is a cut of $\Pi(\mathbb A_3^{\operatorname{bip}}\otimes_k\mathbb A_3^{\operatorname{bip}})$, where $\mathbb A_3^{\operatorname{bip}}$ is the Dynkin quiver of type $A$ with bipartite orientation and $\Pi$ denotes the higher preprojective algebra. 

\begin{ack}
We thank Martin Herschend for pointing out a mistake in the statement of Proposition \ref{prop:rad_no_ext} when the results were first announced at the fd-seminar in June 2020, and we thank Steffen Oppermann and Øyvind Solberg for carefully reading the manuscript. 
\end{ack}


\subsection*{Setup}
Let $k$ be an algebraically closed field. The $k$-dual $\Hom_k(-, k)$ is denoted by $D$. Unless specified otherwise, all modules are left modules. The idempotent associated to a vertex $i$ is denoted by $e_i$. If $a$ and $b$ are arrows in a quiver, then $ab$ denotes the path $b$ followed by $a$. The head of an arrow $a \colon i \rightarrow j$ is denoted by $h(a)$ and equals $j$, and the tail is denoted by $t(a)$ and equals $i$. These extend to paths $p = p_{\ell} p_{\ell-1} \cdots p_1$ by letting $h(p) = h(p_{\ell})$ and $t(p) = t(p_1)$. Moreover, the length of a path $p = p_{\ell} p_{\ell-1} \cdots p_1$ is $\ell$ and this is denoted by $L(p)$.
The syzygy of a module $N$ is the kernel of the projective cover of $N$ and this is denoted by $\Omega N$. 
If $\Lambda$ is a $k$-algebra, then $\modu\Lambda$ denotes the category of finitely generated left modules and $\bd(\modu\Lambda)$ the bounded derived category. When $\Lambda = kQ/I$ is a basic algebra, we always assume that $Q$ is a connected quiver.


\section{Preliminaries}


\subsection{$n$-hereditary algebras}
Let $\Lambda$ be a finite-dimensional algebra of global dimension $n$. Let 
\[
	\Se:= D\Lambda\Lotimes_{\Lambda}- : \bd(\mathsf{mod}\, \Lambda)\to \bd(\mathsf{mod}\, \Lambda)
\]
be the Serre functor with inverse 
\[
	\Se^{-1}=\RHom_{\Lambda}(D\Lambda, -): \bd(\mathsf{mod}\, \Lambda)\to \bd(\mathsf{mod}\, \Lambda).
\]
Denote by $\Se_n$ the composition $\Se_n:= \Se\circ [-n]$.  

\begin{definition}
\label{def:hereditary}
We say that $\Lambda$ is 
\begin{itemize}
\item \emph{$n$-representation-finite ($n$-RF)} if for any indecomposable projective $P\in\proj \Lambda$, there exists $i\geq 0$ such that $\Se_n^{-i}(P)\in \inj \Lambda$, the category of finitely generated injective modules.  
\item \emph{$n$-representation-infinite ($n$-RI)} if $\Se_n^{-i}(\Lambda)\in\modu \Lambda$ for any $i\geq 0$.  
\item \emph{$n$-hereditary} if $\Ho^j(\Se_n^i(\Lambda)) = 0$ for all $i,j\in \Z$ such that $j\not\in n\Z$. 
\end{itemize}
\end{definition}

These definitions, as written, were given in \cite{HIO14}, but the concept of $n$-RF algebras was studied before in \cite{Iya07, HI11b, HI11, Iya11, IO11, IO13}. 

We have the following dichotomy.

\begin{theorem}[{\cite[Theorem 3.4]{HIO14}}]
Let $\Lambda$ be a ring-indecomposable $k$-algebra. Then $\Lambda$ is $n$-hereditary if and only if it is either $n$-RF or $n$-RI.
\end{theorem}

Recall that hereditary algebras $\Lambda$ are \emph{formal}, that is, for any $X\in\bd(\modu \Lambda)$, there is an isomorphism
\[
	X\cong \bigoplus_{i\in \Z} \Ho^j(X)[-j].
\] 
An important feature of $n$-hereditary algebras is that a certain generalisation of this property holds. This follows from {\cite[Lemma 5.2]{Iya11}}. 

\begin{proposition}
Let $\Lambda$ be an $n$-hereditary algebra. Then for any $i\in \Z$ and an indecomposable projective module $P\in\proj \Lambda$, there exists $j\in\Z$ such that
\[
	\Se_n^i(P)\cong \Ho^{nj}(\Se_n^i(P))[-nj].
\]
\end{proposition}

As a consequence, $n$-hereditary algebras satisfy a condition which is closely related to the \emph{vosnex (``vanishing of small negative extensions'')} property (see \cite[Notation 3.5]{IO13}).

\begin{corollary}
\label{cor:vosnex}
Let $\Lambda$ be an $n$-hereditary algebra. Then 
\begin{equation}
\label{eq:vosnex}
	\Ext^{\ell}_{\Lambda}(D\Lambda, \Lambda) = 0
\end{equation}
for all $0<\ell<n$.
\end{corollary}

We refer to this property as the \emph{vanishing-of-$\Ext$} condition. 

As for classical hereditary algebras, preprojective algebras play an important role. 

\begin{definition}
Let $\Lambda$ be a finite-dimensional algebra of global dimension $n$. The \emph{$(n+1)$-preprojective algebra} $\Pi(\Lambda)$ is defined as 
\[
	\Pi(\Lambda):= T_{\Lambda}\Ext_{\Lambda}^n(D\Lambda, \Lambda) \cong \bigoplus_{\ell\geq 0} \Ho^0(\Se_n^{-\ell}(\Lambda)).
\]
\end{definition} 

Note that $\Ext_{\Lambda}^\ell(D\Lambda, \Lambda)\cong\Ext_{\Lambda^e}^{\ell}(\Lambda, \Lambda^e)$ \cite[Lemma 2.9]{GI19}, a fact that we use often. 

Preprojective algebras and $n$-hereditary algebras are connected in the following way.

\begin{theorem}
Let $\Lambda$ be a finite-dimensional algebra.
\begin{enumerate}
\item If $\Lambda$ is an $n$-representation-finite algebra. Then $\Pi(\Lambda)$ is a selfinjective algebra. The converse holds if $\Lambda$ has global dimension $2$. 
\item The following are equivalent. 
	\begin{enumerate}[a)]
	\item $\Lambda$ is $n$-representation-infinite;
	\item $\Pi(\Lambda)$ is a bimodule Calabi--Yau algebra of Gorenstein parameter $1$. 
	\end{enumerate}
\end{enumerate}
\end{theorem}

Here, \textit{(1)} is due to \cite[Corollary 3.4 \& Corollary 3.8]{IO13}, whereas \textit{(2)} is an amalgam of results from \cite[Theorem 4.8]{Kel11}, \cite[Corollary 4.13]{MM11}, \cite[Theorem 4.36]{HIO14}, and \cite[Theorem 3.4]{AIR15}. We refer to the papers for definitions. 

In the case where $\Lambda$ is Koszul, we have a good understanding of the construction of preprojective algebras. 
To present the construction we need certain notions of derivatives which we define below, and we note that they are used extensively in this paper, and not just in the context of Koszul algebras.


\subsection*{Notation for derivatives}

Let $S$ be a semisimple $k$-algebra and $V$ be an $S$-bimodule.
Let $p = v_{\ell}\otimes\cdots\otimes v_1\in V^{\otimes_S \ell}$. We define the linear morphisms
 \[
 	\delta_m^{\Le}(p) := v_{\ell-m}\otimes \cdots \otimes v_1\quad\text{and}\quad\delta_m^{\Ri}(p):= v_{\ell}\otimes\cdots\otimes v_{m+1}.
 \]
 for $m < \ell$ and we let both equal $0$ when $\ell = m$. 
 
Moreover, we define 
\[
	\Le_m(p) := v_{\ell}\otimes\cdots\otimes v_{\ell-m+1} =\delta_{\ell-m}^{\Ri}(p) \quad\text{and}\quad\Ri_m(p) = v_m\otimes\cdots \otimes v_1 = \delta_{\ell-m}^{\Le}(p).
\]
The subscript is dropped if $m=1$.

We also define linear morphisms associated to elements $q\in V^{\otimes m}$, for $m\leq \ell$:
\[
	\delta^{\mathcal L}_q (p) := \left\{
	\begin{array}{ll}
		a  & \mbox{if } p = q\otimes a \\
		0 & \mbox{else}
	\end{array}
\right. \quad \text{and}\quad \delta^{\mathcal R}_q (p) := \left\{
	\begin{array}{ll}
		b  & \mbox{if } p = b\otimes q \\
		0 & \mbox{else}.
	\end{array}
\right.
\]
Similarly, we define 
\[
	\mathcal L_q (p) := \left\{
	\begin{array}{ll}
		b  & \mbox{if } p = b\otimes q\otimes a \\
		0 & \mbox{else} 
	\end{array}
\right. \quad \text{and}\quad \mathcal R_q (p) := \left\{
	\begin{array}{ll}
		a  & \mbox{if } p = b\otimes q\otimes a \\
		0 & \mbox{else}.
	\end{array}
\right.
\]
When $p = b\otimes q\otimes a$ for some paths $a$ and $b$, we say that $q$ \emph{divides} $p$ and denote this by $q | p$. 

\subsection*{Description of the $n$-preprojective algebra of a Koszul $n$-hereditary algebra}
Recall that if $\Lambda$ is Koszul, it can be given as a tensor algebra $T_SV/\langle M \rangle$ where, as in the previous section, $S$ is some semisimple $k$-algebra, $V$ is an $S$-bimodule, and $M \subset V \otimes_S V$ is a subbimodule \cite{BGS96}.
Let then
\[
K_{\ell}:=\bigcap_{\mu=0}^{\ell-2} (V^{\otimes\mu}\otimes M \otimes V^{\otimes \ell-\mu-2})
\]
be the terms appearing in the minimal Koszul resolution of $\Lambda$ according to \cite{BGS96}. Moreover, given a vector space $V$, let $\mathcal B(V)$ be a basis. 

\begin{proposition}[{\cite[Proposition 3.12]{GI19}, \cite[Corollary 3.3]{Thi20}}]
\label{thm:quiver_construction}
Let $\Lambda = T_S V/\langle M\rangle$ be a finite-dimensional Koszul algebra of global dimension $n$. Let $\{e_i \, |\, 1\leq i \leq m\}$ be a complete set of primitive orthogonal idempotents in $\Lambda$. Let $\overline V$ be the vector space obtained from $V$ by adding a basis element $e_i a_q e_j$ for each element $q\in \mathcal B(e_jK_ne_i)$. Let $\overline M$ be the union of $M$ with the set $\widetilde M$ of quadratic relations given by
\[
	\widetilde M:=  \left\{\sum_{q\in \mathcal B(K_n)}a_q \delta^{\mathcal R}_p (q)+ (-1)^n \sum_{q\in \mathcal B(K_n)}\delta_p^{\mathcal L} (q) a_q \quad | \quad p\in \mathcal B(K_{n-1}) \right\}.
\] 
There is an isomorphism of algebras
\[
	\Pi\cong T_S \overline V/\langle \overline M\rangle.
\]
\end{proposition}


\subsection{Monomial algebras}
\label{sec:mono} 
In this subsection, we define monomial algebras and describe certain minimal projective resolutions.  

\begin{definition}
Let $\Lambda = kQ/I$, where $Q$ is a finite quiver and $I$ an admissible ideal. We say that $\Lambda$ is a \emph{monomial algebra} if $I$ can be generated by a finite number of paths. 
\end{definition}

There is a nice description of the minimal projective $\Lambda$-bimodule resolution of $\Lambda$, due to Bardzell \cite{Bar97}. Let $M$ be a minimal set of paths of minimal length which generates $I$. 
Given a path $p$, define the \emph{support} to be the set of all vertices dividing $p$. For every directed path $T$ in $Q$, there is a natural order $<$ on the support of $T$.  Let $M(T)$ be the set of relations which divide $T$.

\begin{definition}
Let $p\in M(T)$. We define the \emph{left construction associated to $p$ along $T$} by induction. Let $r_2\in M(T)$ be the path (if it exists) in $M(T)$ which is minimal with respect to $t(p)<h(r_2) < h(p)$. Now assume we have constructed $r_1 = p, r_2, \ldots, r_j$. Let 
\[
	L_{j+1} = \{r\in M(T) \,\,|\,\, h(r_{j-1}) \leq t(r) < h(r_j)\}.
\]
If $L_{j+1} \not = \varnothing$, let $r_{j+1}$ be such that $t(r_{j+1})$ is minimal in $L_{j+1}$. 
\end{definition}

\begin{definition}
Let $p\in M$ and $\ell\geq 2$ be an integer. We define 
\begin{multline*}
	\ASb_p(\ell) := \{(r_1 = p, r_2, \ldots, r_{\ell-1})\,\, | \,\, (r_1, r_2, \ldots, r_{\ell-1}) \text{ is a sequence of paths associated}\\ \text{to $p$ in the left construction}\}.
\end{multline*}
For each element $(r_1, \ldots, r_{\ell-1})\in \ASb_p(\ell)$, define $p^\ell$ to be the path from $t(p)$ to $h(r_{\ell-1})$ and let $\APb_p(\ell)$ be the set of all $p^\ell$. Finally, we define 
\[
	\APb(\ell) := \bigcup_{p\in M} \APb_p(\ell),
\] 
if $\ell\geq 2$ and $\APb(0) := Q_0$, $\APb(1):= Q_1$. 
\end{definition}

The vector spaces $k\APb(\ell)$ are the $kQ_0$-bimodules which appear in the minimal resolution we want to construct. Note that $\APb(2) = M$. If $p\in \APb(\ell)$, define
\[
	\Sub(p) := \{q\in \APb(\ell-1)\,\, |\,\, q \text{ divides } p\}.
\]	

\begin{lemma}[{\cite[Lemma 3.3]{Bar97}}]
The set $\Sub(p)$ contains two paths $p_0$ and $p_1$ such that $t(p_0) = t(p)$ and $h(p_1) = h(p)$. Moreover, if $\ell$ is odd, then $\Sub(p) = \{p_0, p_1\}$. 
\end{lemma}

We are now ready to define morphisms 
\[
	d_{\ell}: \bigoplus_{p\in \APb(\ell)}\Lambda e_{h(p)}\otimesk e_{t(p)}\Lambda \to \bigoplus_{p\in \APb(\ell-1)}\Lambda e_{h(p)}\otimesk e_{t(p)} \Lambda,
\]
noting that we give our conventions with respect to idempotents, and heads and tails of arrows and paths in the setup immediately following the introduction.
Recall that if $p\in \APb(\ell)$ and $q\in \Sub(p)$, we write $p = \Le_q(p)q\Ri_q(p)$. By the previous lemma, we have that $\Sub(p) = \{p_0, p_1\}$ if $\ell$ is odd, in which case $p = \Le_{p_0}(p) p_0$ and $p = p_1 \Ri_{p_1}(p)$. Then we define 
\[
	d_{\ell}((e_{h(p)}\otimes e_{t(p)})_p) := \begin{cases}
       (\Le_{p_0}(p)e_{h(p_0)}\otimes e_{t(p_0)})_{p_0}  - (e_{h(p_1)}\otimes e_{t(p_1)} \Ri_{p_1}(p))_{p_1} &\quad\text{if $\ell$ is odd}\\
       \sum_{q\in\Sub(p)} (\Le_q(p)e_h\otimes e_t \Ri_q(p))_q &\quad\text{if $\ell$ is even}.
     \end{cases}
\]
Here, we use the notation $(-\otimes-)_p$ to denote an element in the $p$-th component in $\bigoplus_p \Lambda e_i\otimesk e_j\Lambda$.

\begin{theorem}[{\cite[Theorem 4.1]{Bar97}}]
\label{thm:Bardzell}
The complex 
\begin{equation}
\label{eq:Bardzell}
	\cdots\xrightarrow{d_{n+1}}  \bigoplus_{p\in \APb(n)}\Lambda e_{h(p)}\otimesk e_{t(p)}\Lambda\xrightarrow{d_n}\cdots\xrightarrow{d_1} \bigoplus_{e_i\in \APb(0)}\Lambda e_i\otimesk e_i\Lambda \xrightarrow{\mu} \Lambda\to 0,
\end{equation}
where $\mu((e_i\otimes e_i)_{e_i}) = e_i$, is a minimal projective resolution of $\Lambda$ as a $\Lambda$-bimodule.
\end{theorem}

\subsection{Computing $\Ext_{\Lambda^e}^{\ell}(\Lambda, \Lambda^e)$} 
\label{sec:comp}

In the next sections, we use on many occasions Corollary \ref{cor:vosnex} as an obstruction for certain algebras to be $n$-hereditary. 
We therefore explain here how to compute $\Ext_{\Lambda^e}^{\ell}(\Lambda, \Lambda^e)$ for $1\leq \ell\leq n$.

Let $\Lambda$ be a basic finite-dimensional algebra. By \cite[Section 1.5]{Hap89}, $\Lambda$ has a minimal projective bimodule resolution of the form 
\[
	P_{\bullet}: \cdots\xrightarrow{d_{n+1}}  \bigoplus_{p\in \mathcal B(E^{n}(i,j))}\Lambda e_{h(p)}\otimesk e_{t(p)}\Lambda\xrightarrow{d_n}\cdots\xrightarrow{d_1} \bigoplus_{e_i\in \mathcal B(E^0(i,j))}\Lambda e_i\otimesk e_i\Lambda \to 0,
\]
where $E^{\ell}(i,j): = \Ext^{\ell}_{\Lambda}(S_i, S_j)$ and $S_i$ denotes the simple module at vertex $i$. In the case where $\Lambda$ is monomial, we have $E^{\ell}(i,j)\cong e_jk\APb(\ell)e_i$. Note that, in general, it is hard to determine the differentials $d_{\ell}$. 

In order to compute $\Ext_{\Lambda^e}^{\ell}(\Lambda, \Lambda^e)$, we apply $\Hom_{\Lambda^e}(-, \Lambda^e)$ to $P_{\bullet}$ and use the isomorphisms 
\[
\begin{array}{ccccc}
	\Psi: \Hom_{\Lambda^e}(\Lambda e_j\otimesk e_i\Lambda, \Lambda^e) &\cong& e_j \Lambda\otimesk \Lambda e_i & \cong& \Lambda e_i\otimesk e_j\Lambda \\
	\phi & \mapsto& \phi(e_j\otimes e_i)  && \\
	&& e_j\otimes e_i &\mapsto & e_i\otimes e_j 
\end{array}
\]
to obtain a complex 
\begin{equation}
\label{eq:complex}
	\Hom_{\Lambda^e}(P_{\bullet}, \Lambda^e): 0\to \bigoplus_{e_i\in \mathcal B(E^0(i,j))}\Lambda e_i\otimesk e_i\Lambda \xrightarrow{\tilde d_1} \cdots\xrightarrow{\tilde d_n}\bigoplus_{p\in \mathcal B(E^{n}(i,j))}\Lambda e_{t(p)}\otimesk e_{h(p)}\Lambda\to\cdots,
\end{equation}
where $\tilde d_{\ell}(e_i\otimes e_j) = \Psi(\Psi^{-1}(e_i\otimes e_j)\circ d_{\ell})$.

Computing $\Ext_{\Lambda^e}^{\ell}(\Lambda, \Lambda^e)$ requires the understanding of the morphisms $\tilde d_{\ell}$, which we do have in the case where $\Lambda$ is monomial. In fact, we have    
\[
	\tilde d_{\ell}((e_{t(p)}\otimes e_{h(p)})_p) = \begin{cases}
       \sum_{q\in \APb(\ell)} (e_{t(q)}\otimes e_{h(q)} \delta^{\mathcal R}_p (q))_{q} -\sum_{q\in \APb(\ell)}  (\delta^{\mathcal L}_p (q) e_{t(q)}\otimes e_{h(q)})_{q} &\,\,\text{if $\ell$ is odd}\\
       \sum_{q\in \APb(\ell) \,|\, p\in\Sub(q)} (\Ri_p(q)e_{t(q)}\otimes e_{h(q)} \Le_p(q))_q &\,\,\text{if $\ell$ is even}.
     \end{cases}
\]

In further sections, we use these to describe cocycles and coboundaries, allowing us to show that some  $\Ext_{\Lambda^e}^{\ell}(\Lambda, \Lambda^e)$ does not vanish for some algebra, thus preventing them from being $n$-hereditary. Using Corollary \ref{cor:vosnex}, we can already give a necessary condition for a monomial algebra to be $n$-hereditary. This is analogous to results established in \cite[Proof of Theorem 3.14]{GI19} and \cite[Proof of Theorem 3.6]{Thi20} in the case where $\Lambda$ is Koszul.

\begin{lemma}
\label{lem:deltaK_n=K_{n-1}}
Let $\Lambda$ be a monomial algebra and define
\[
	\delta(E(i,j)^{\ell}):= \{w\in E(i,j)^{\ell-1}\,\, |\,\, w\in \Sub(w')\text{ for some } w'\in E(i,j)^{\ell}\}.
\]
Then $E(i,j)^{\ell-1}= \delta(E(i,j)^{\ell})$ for all $2\leq \ell \leq n$.
\end{lemma}

\begin{proof}
Suppose by contradiction that there exists $w\in E(i,j)^{\ell-1}$ which does not divide any element of $E(i,j)^{\ell}$, for some $\ell$. Then 
\[
	\tilde d_{\ell}((e_{t(w)}\otimes e_{h(w)})_w) = 0,
\]
which means that $(e_{t(w)}\otimes e_{h(w)})_w$ is an $(\ell-1)$-cocycle in $\Hom_{\Lambda^e}(P_{\bullet}, \Lambda^e)$. However, it is not a coboundary, implying that $\Ext^{\ell-1}_{\Lambda^e}(\Lambda, \Lambda^e)\not =0$. This contradicts Corollary \ref{cor:vosnex}.
\end{proof}

\begin{corollary}
\label{cor:arrow_in_rel}
Let $\Lambda = kQ/\langle M\rangle$ be a basic $n$-hereditary monomial algebra where $M$ is a set of relations given by paths in $kQ$. Then every arrow in $Q_1$ is part of a relation in $M$. 
\end{corollary}

\begin{proof}
By Lemma \ref{lem:deltaK_n=K_{n-1}}, we have that $\delta(M) = V$. 
\end{proof}

\section{Classification of $n$-hereditary truncated path algebras}
\label{sec:trunc}

In this section, we assume that $\Lambda = kQ/I$ is a monomial finite-dimensional algebra, that is, the ideal $I$ in a presentation of $\Lambda$ can be chosen to be generated by paths. Moreover, for the rest of the text, whenever $\Lambda$ is assumed to be monomial, we also assume $I = \langle M \rangle$ with $M$ a minimal set of paths of minimal length. 

Recall that, by the vanishing-of-$\Ext$ condition, we have that $\Ext_{\Lambda^e}^i(\Lambda, \Lambda^e) = 0$ for all $0<i<n$ for any $n$-hereditary algebra $\Lambda$. We thus seek to understand what knowledge one can obtain from this property. As an application, we use this information to classify the truncated path algebras $\Lambda =kQ/\mathcal J^{\ell}$, where $\mathcal J$ is the arrow ideal, which are $n$-hereditary in the second subsection.


\subsection{Vanishing-of-$\Ext$ condition for monomial algebras}
 In this subsection, we find necessary conditions on the quiver and relations of monomial path algebras in order to satisfy the vanishing-of-$\Ext$ condition. To be more precise, we only look into the vanishing of the first $\Ext$. Recall that, by Lemma \ref{lem:deltaK_n=K_{n-1}}, every arrow has to be part of at least one relation, otherwise $\Ext_{\Lambda^e}^1(\Lambda, \Lambda^e)\not = 0$. This is a first obstruction, which does not require the monomial hypothesis. We therefore assume this property for the class of algebras we consider in this subsection. Throughout, we let $\Lambda$ be a monomial algebra in which every arrow divides at least one relation. 

The main strategy is to construct cocycle elements which are not coboundaries in the complex (\ref{eq:complex}), defined as $\Hom_{\Lambda^e}(P_{\bullet}, \Lambda^e)$, where $P_{\bullet}$ is the minimal projective $\Lambda$-bimodule resolution of $\Lambda$, described in the preliminaries. We refer to Sections \ref{sec:mono} and \ref{sec:comp} for more details and the notation.

\begin{proposition}
\label{prop:arrow_start_end}
Suppose that there exists an arrow $a$ which is the start (resp.\ the end) of every relation it divides and such that $t(a)$ (resp. $h(a)$) is not a source (resp.\ a sink). Then $\Ext^1_{\Lambda^e}(\Lambda, \Lambda^e)\not = 0$. 
\end{proposition}

\begin{proof}
Assume that there is an arrow $a$ which is the end of every relation $r_i$ it divides and such that $h(a)$ is not a sink. The other case is dual. We consider the element 
\[
	(ae_{t(a)}\otimesk e_{h(a)})_a \in \bigoplus_{v\in Q_1}\Lambda e_{t(v)}\otimesk e_{h(v)}\Lambda
\]
from complex (\ref{eq:complex}). Then 
\[
	\tilde d_2((ae_{t(a)}\otimesk e_{h(a)})_a) = \sum_i (a\mathcal R_1(r_i)e_{t(r_i)}\otimesk e_{h(a)})_{r_i} = 0,
\]
so it is a cocycle in complex (\ref{eq:complex}). However, since $h(a)$ is not a sink, $(ae_{t(a)}\otimesk e_{h(a)})_a$ cannot be a coboundary. In fact, let $b$ be an arrow such that $h(a) = t(b)$. Then, 
\[
	\tilde d_1((e_{h(a)}\otimesk e_{h(a)})_{e_{h(a)}} = (ae_{t(a)}\otimesk e_{h(a)})_a + (e_{h(a)}\otimesk e_{h(b)}b)_b + \ldots
\]
This is the only place where $(ae_{t(a)}\otimesk e_{h(a)})_a$ appears as a summand of an element in the image of $\tilde d_1$. The same is true for $(e_{h(a)}\otimesk e_{h(b)}b)_b$, which means that this term cannot be cancelled by other elements in the image of $\tilde d_1$. Therefore, $(ae_{t(a)}\otimesk e_{h(a)})_a$ is not a coboundary and $\Ext^1_{\Lambda^e}(\Lambda, \Lambda^e)\not = 0$.
\end{proof}
  
 We say that two relations \emph{intersect} with each other if there is at least one arrow which divides both of them. We have the following corollary. 
 
 \begin{corollary}
 Assume that there is a relation $r$ which does not intersect with any other relation and such that $t(r)$ and $h(r)$ are not both a source and a sink. Then $\Ext^1_{\Lambda^e}(\Lambda, \Lambda^e)\not = 0$. 
 \end{corollary} 

Continuing on the same ideas, we explore what happens at sinks and sources. We show that the vanishing-of-Ext conditions implies that sinks and sources divide only one arrow.

\begin{proposition}
\label{prop:nosinksource}
Assume that there is a vertex $i$ in $Q$ which is a sink (resp. a source), such that there is at least two arrows having $i$ as head (resp. as tail). Then $\Ext^1_{\Lambda^e}(\Lambda, \Lambda^e)\not = 0$.
\end{proposition}

\begin{proof}
We suppose that $i$ is a sink. The other case is dual. Let $a$ and $b$ be two arrows such that $h(a) = h(b) = i$. We claim that the element $(ae_{t(a)}\otimesk e_i)_a \in \bigoplus_{v\in Q_1} \Lambda e_{t(v)}\otimesk e_{h(v)} \Lambda$ is a cocycle in degree $1$. In fact, since $h(a)$ is a sink, every relation $r$ containing $a$ is of the form $r = a R_r(a)$ for some path $R_r(a)$. Therefore, 
\[
	\tilde d_2((ae_{t(a)}\otimesk e_i)_a) = \sum_{a | r} (aR_r(a) e_{t(r)}\otimes e_i)_r = 0.
\]
This is however not a coboundary. In fact, since $i$ is also the head of another arrow, we have that 
\[
	\tilde d_1((e_i \otimes e_i)_{e_i}) = (ae_{t(a)}\otimesk e_i)_a + (be_{t(b)}\otimesk e_i)_b + \ldots
\]
By the same reasoning as in Proposition \ref{prop:arrow_start_end}, we conclude that $\Ext^1_{\Lambda^e}(\Lambda, \Lambda^e)\not = 0$.
\end{proof}


\subsection{$n$-hereditary truncated path algebras}

We now consider the case of truncated path algebras $\Lambda =kQ/\mathcal J^{\ell}$ for some $\ell\geq 2$, where $Q$ is a finite quiver and $\mathcal J$ is the arrow ideal. In this case, the terms in the Bardzell resolution (\ref{eq:Bardzell}) are particularly easy to describe. Indeed, the vector space $k\APb(i)$ is generated by all paths of length $\frac{i}{2}\cdot \ell$ if $i$ is even and those of length $\left(\frac{i-1}{2}\cdot\ell +1\right)$ if $i$ is odd. Let $L(p)$ denote the length of a path $p$. We use the following results. 

\begin{theorem}{\cite[Theorem 2]{DHZL08}}
\label{lem:skeleton}
Let $\Lambda$ be a truncated path algebra. If $N$ is a non-zero $\Lambda$-module with skeleton $\sigma$, then the syzygy of $N$
\[
	\Omega N\cong \bigoplus_{q\text{ $\sigma$-critical}}\Lambda q.
\]
\end{theorem}

\noindent We refer to the paper for the definitions of skeletons $\sigma$ and of $\sigma$-critical paths.

We also need the following result regarding extensions of certain kinds of indecomposable modules. 

\begin{proposition}[{\cite[Proposition 3.1]{Vas19}}]
\label{lem:indec_syz}
Let $\Lambda$ be a finite-dimensional algebra. Let $N\in\modu\Lambda$ be a non-projective indecomposable module. If $\Omega\, N$ is decomposable, then $\Ext^1_{\Lambda}(N, \Lambda)\not=0$.
\end{proposition}

Let now $\mathbb A_m$ be the linearly oriented Dynkin quiver of type $A$ 

\begin{center}
\begin{tikzpicture}[auto, baseline=(current  bounding  box.center)]
	\node[circle, draw, thin,fill=black!100, scale=0.4] (0) at (0,0) {};
	\node[circle, draw, thin,fill=black!100, scale=0.4] (1) at (1.414,0) {};
	\node[circle, draw, thin,fill=black!100, scale=0.4] (2) at (2.828,0) {};
	\node[circle, draw, thin,fill=black!100, scale=0.4] (3) at (4.242,0) {};
	\node[circle, draw, thin,fill=black!100, scale=0.4] (4) at (5.656,0) {};
	\node[circle, draw, thin,fill=black!100, scale=0.4] (5) at (7.07,0) {};
	\draw[decoration={markings,mark=at position 1 with {\arrow[scale=2]{>}}},
    postaction={decorate}, shorten >=0.4pt] (0) to (1);
	\draw[decoration={markings,mark=at position 1 with {\arrow[scale=2]{>}}},
    postaction={decorate}, shorten >=0.4pt] (1) to (2);
	\draw[dashed] (3,0) to (4,0);
	\draw[decoration={markings,mark=at position 1 with {\arrow[scale=2]{>}}},
    postaction={decorate}, shorten >=0.4pt] (3) to (4);
    	\draw[decoration={markings,mark=at position 1 with {\arrow[scale=2]{>}}},
    postaction={decorate}, shorten >=0.4pt] (4) to (5);
\end{tikzpicture}
\end{center}
with $m$ vertices.

\begin{proposition}
\label{prop:rad_no_ext}
Let $Q$ be a finite acyclic quiver and assume that $Q\not = \mathbb A_m$. Let $\Lambda := kQ/\mathcal J^\ell$ for some $\ell\geq 2$ be a truncated path algebra. Then there exists $0<j<\gldim \Lambda$ such that $\Ext^j_{\Lambda^e}(\Lambda, \Lambda^e) \not = 0$. 
\end{proposition}
\begin{proof}
By Proposition \ref{prop:nosinksource}, if $Q$ is a Dynkin quiver of type $A$ with a non-linear orientation, then $\Ext^1_{\Lambda^e}(\Lambda, \Lambda^e)\not = 0$. Since $Q\not=\mathbb A_m$, there exists a vertex $i$ which divides at least $3$ arrows. If $i$ is either a source or a sink, then $\Ext^1_{\Lambda^e}(\Lambda, \Lambda^e)\not = 0$ by Proposition \ref{prop:nosinksource} as well. 

Now suppose that $i$ is the head of at least two arrows and the tail of at least one arrow. The opposite case is treated similarly. Among the arrows with head $i$ we pick two, say, $a_r$ and $b_s$ satisfying that $a_r \neq b_s$ and that there exist paths $T_1:= a_r\cdots a_1$ and $T_2:=b_s\cdots  b_1$ which are maximal in the following sense: without loss of generality, we let $T_1$ be the longest path in $kQ$ ending at $i$ and $T_2$ the maximal path in $kQ$ ending at $i$ not divided by $a_r$. Note that this uses that $Q$ is acyclic.
In particular, we assume $L(T_2) \leq L(T_1)$. 
Moreover, we let $T_3 := c_t \cdots c_1$ be the longest path in $kQ$ beginning in $i$. 

We may also assume that $h(T_3)$ is only a sink to the arrow $c_t$ and $t(T_i)$ is a source to only one arrow for $i=1,2$, since otherwise $\Ext^1_{\Lambda^e}(\Lambda, \Lambda^e)\not = 0$ and we are done.

We split the proof into the following cases: 
\begin{itemize}
	\item [\textbf{C1}:] $L(T_3T_2) \leq \ell-1$
	\item [\textbf{C2}:] $L(T_3T_2)\geq \ell$
		\begin{itemize}
			\item [\textbf{a)}] $L(T_2)\leq \ell -1 $
			\item [\textbf{b)}] $L(T_2)\geq \ell$
		
			\begin{itemize}
				\item [\textbf{i)}] $L(T_3) \geq \ell -1$
				\item [\textbf{ii)}] $L(T_3) \leq \ell -2$\\
			\end{itemize}
		\end{itemize}
\end{itemize}

\noindent\textbf{C1}: If $L(T_3T_2)\leq \ell-1$, then $b_s$ does not divide any relation, by the maximality assumption on the length of $T_3T_2$. 
As a consequence of Lemma \ref{lem:deltaK_n=K_{n-1}}, $\Ext^1_{\Lambda^e}(\Lambda, \Lambda^e)\not = 0$ and we are done.\\

 \noindent\textbf{C2}: Now suppose that $L(T_3T_2)\geq \ell$.\\
 
 \noindent\textbf{a)} If $L(T_2)\leq\ell-1$, then for any relation path $p$ (of length $\ell$) such that $b_s$ divides $p$, we have that 
 \[
 	L(\mathcal L_{b_s}(p))\geq \ell - s\geq \max(1, \ell-r),
\]
where $r := L(T_1)$ and $s:=L(T_2)$. The first inequality is explained by the maximality assumption on $L(T_3T_2)$.  
For the second inequality, recall that we have assumed without loss of generality that $r\geq s$. This means that any path of maximal length starting at $i$ is of length at least $\max(1, \ell-r)$. Therefore, the element 
 \[
 	(e_{t(b_s)}\otimesk e_ia_r\cdots a_{\max(1, r-\ell+2)})_{b_s}\in\bigoplus_{v\in Q_1}\Lambda e_{t(v)}\otimesk e_{h(v)}\Lambda
 \]
is a non-trivial cocycle. 
It is not a coboundary since the only two $\Lambda$-bimodule generators in\\ $\bigoplus_{i\in Q_0}\Lambda e_{i}\otimesk e_{i}\Lambda$ which map non-trivially via $\tilde d_1$ to an element in $\Lambda e_{t(b_s)}\otimesk e_i\Lambda$ are 
\[
	(e_{t(b_s)}\otimesk e_{t(b_s)})_{e_{t(b_s)}}\mapsto (e_{t(b_s)}\otimesk e_i b_s)_{b_s} + \ldots
\]
and 
\[
	(e_{i}\otimesk e_{i})_{e_{i}}\mapsto (b_s e_{t(b_s)}\otimesk e_i)_{b_s} + \ldots
\]
and they cannot be linearly combined to obtain our cocycle. Thus, $\Ext^1_{\Lambda^e}(\Lambda, \Lambda^e)\not = 0$. \\

\noindent\textbf{b.i)} We now consider the case where $L(T_2)\geq\ell$. Let $j\in\N_{\geq 1}$ be such that the length of the paths in $k\APb(j)$ is less than or equal to $L(T_2)$, but the length of the paths in $k\APb(j+1)$ is strictly bigger than $L(T_2)$. 
If $L(T_3)\geq \ell-1$, then $0<j<\gldim\Lambda$, since $k\APb(j+1)$ is non empty, as it contains a path dividing $T_3T_2$. 
Let $T:=b_s\cdots b_x$ be the path in $k\APb(j)$ ending at $i$ and dividing $T_2$. 
Then the element 
\[
	(e_{t(T)}\otimesk e_{i}a_r\cdots a_{r-\ell+2})_T\in \bigoplus_{p\in \APb(j)}\Lambda e_{t(p)}\otimesk e_{h(p)}\Lambda
\]
is a cocycle. 
Indeed, for any $T'\in \APb(j+1)$ which is divided by $T$, we have $L(\mathcal L_{T}(T'))\geq 1$. 
This is explained by the fact that $L(T')>L(T_2)$ and the maximality assumption on the length of $T_2$. 
In fact, if $L(\mathcal L_{T}(T'))=0$, then $h(T') = i$ and $L(T_3T') > L(T_3T_2)$, contradicting our hypothesis. 

The element is not a coboundary for a similar reason as above. 
\\

\noindent\textbf{b.ii)} Now, if $L(T_3)\leq\ell-2$, then we consider the indecomposable injective module $I$ associated to the vertex $h:=h(T_3)$. 
We show that either there are more than one $\sigma$-critical paths or there is only one $\sigma$-critical path $q$ and $\Lambda q$ is projective. 
In the former case, we conclude by Theorem \ref{lem:skeleton} and Proposition \ref{lem:indec_syz} that $\Ext^1_{\Lambda^e}(\Lambda, \Lambda^e)\cong \Ext^1_{\Lambda}(D\Lambda, \Lambda) \not =0$. 
In the latter case, we obtain the same conclusion since $\projdim I = 1$.

We call \emph{branching points} the vertices which divide at least $3$ arrows. Let $\mathcal S_h$ be the support of paths of length $\ell-1$ in $kQ$ which end in $h$. Let $\mathcal B_h$ be the set of branching points which are in $\mathcal S_h$. Since $L(T_3)\leq\ell-2$, we have that $i\in\mathcal B_h$. 

Let $\mathcal  S'_h\subset \mathcal S_h$ be the set of vertices which start the paths of length $\ell-1$ that end in $h$. Note that $P:=\bigoplus_{\iota\in \mathcal S'_h} \Lambda e_{\iota}^{m(\iota)}$ is the projective cover of $I$, where $m(\iota)$ is the number of paths of length $\ell$ which ends in $h$ and starts in $\iota$. Because $L(T_3)<\ell-1$, we have $|\mathcal S'_h|\geq 2$, since it contains a vertex in $T_1$ and $T_2$. Thus, $I$ is not a projective module.

Let $x\in B_h$ be such that there exist arrows $\alpha$ and $\beta$ ending in $x$. Then
either paths of the form $\alpha p$ or of the form $\beta q$ are in the skeleton $\sigma$, for $p,q \in \sigma$. The paths not in $\sigma$ must then be $\sigma$-critical. In fact, they get identified via  $P\twoheadrightarrow I$. Thus, every such branching point gives rise to $\sigma$-critical paths. 

Now suppose that there exists a vertex $x\in \mathcal B_h$ which is the start of an arrow $\alpha$ not in a path of length $\ell-1$ ending in $h$. Then for any skeleton $\sigma$, we have that any path of the form $\alpha p$, for $p\in \sigma$, is $\sigma$-critical, since it goes to $0$ via $P\twoheadrightarrow I$. 

Therefore, in order to have only one $\sigma$-critical path, it is necessary that the full subquiver $\bar Q$ containing all the directed paths connected to the branching points in $\mathcal B_h$ is given by
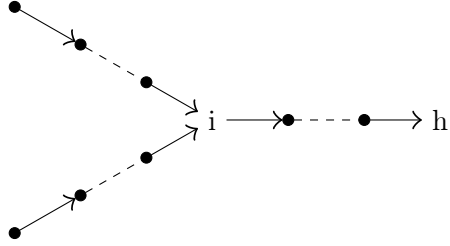
\begin{center}
\begin{tikzpicture}[auto, baseline=(current  bounding  box.center)]
	\node[circle, draw, thin,fill=black!100, scale=0.4] (0) at (-2.598, 1.5) {};
	\node[circle, draw, thin,fill=black!100, scale=0.4] (1) at (-1.732,1) {};
	\node[circle, draw, thin,fill=black!100, scale=0.4] (2) at (-0.866,0.5) {};
	\node[] (3) at (0,0) {i};
	\node[circle, draw, thin,fill=black!100, scale=0.4] (4) at (-2.598,-1.5) {};
	\node[circle, draw, thin,fill=black!100, scale=0.4] (5) at (-1.732,-1) {};
	\node[circle, draw, thin,fill=black!100, scale=0.4] (6) at (-0.866,-0.5) {};
	\node[circle, draw, thin,fill=black!100, scale=0.4] (7) at (1,0) {};
	\node[circle, draw, thin,fill=black!100, scale=0.4] (8) at (2,0) {};
	\node[] (9) at (3,0) {h};

	\draw[decoration={markings,mark=at position 1 with {\arrow[scale=2]{>}}},
   postaction={decorate}, shorten >=0.4pt] (0) to (1);
	\draw[dashed] (1) to (2);
    	\draw[decoration={markings,mark=at position 1 with {\arrow[scale=2]{>}}},
   postaction={decorate}, shorten >=0.4pt] (2) to (3);
	\draw[decoration={markings,mark=at position 1 with {\arrow[scale=2]{>}}},
    postaction={decorate}, shorten >=0.4pt] (4) to (5);
	\draw[dashed] (5) to (6);
	\draw[decoration={markings,mark=at position 1 with {\arrow[scale=2]{>}}},
    postaction={decorate}, shorten >=0.4pt] (6) to (3);

	\draw[decoration={markings,mark=at position 1 with {\arrow[scale=2]{>}}},
   postaction={decorate}, shorten >=0.4pt] (3) to (7);
	\draw[dashed] (7) to (8);
    	\draw[decoration={markings,mark=at position 1 with {\arrow[scale=2]{>}}},
    postaction={decorate}, shorten >=0.4pt] (8) to (9);

\end{tikzpicture}
\end{center}
In this case, we have that $\Omega I \cong \Lambda e_i$ is projective. \end{proof}

The $n$-representation-finite Nakayama algebras were classified by Vaso in \cite{Vas19}. Using his classification, we obtain as a corollary of the previous proposition a classification of all $n$-hereditary algebras of the form $\Lambda = kQ/\mathcal J^{\ell}$.  

\begin{theorem}
\label{thm:class_trunc}
Let $\Lambda = kQ/\mathcal J^{\ell}$ for some $\ell\geq 2$ and finite quiver $Q$. The following are equivalent.
\begin{enumerate}
\item $\Lambda$ is $n$-hereditary;
\item $\Lambda \cong k\mathbb A_m/\mathcal J^{\ell}$, for some $m$, and $\ell \,|\, m-1$ or $\ell =2$.
\end{enumerate}
In this case, $n = 2\frac{m-1}{\ell}$ and $\Lambda$ is a Nakayama $n$-representation-finite algebra. 
\end{theorem}

\begin{proof}
By \cite[Theorem 5]{DHZL08}, any truncated path algebra of finite global dimension must have an acyclic quiver. By Proposition \ref{prop:rad_no_ext}, if $\Lambda$ is $n$-hereditary, then its quiver must be $\mathbb A_m$, since $n$-hereditary algebras satisfy the property that $\Ext^i_{\Lambda^e}(\Lambda, \Lambda^e) =0$ for all $0<i<n$. Therefore, $\Lambda$ is an  $n$-representation-finite Nakayama algebra. The result thus follows from \cite[Theorem 3]{Vas19}.
\end{proof}


\section{Classification of $n$-hereditary quadratic monomial algebras} 

In this section, we give a partial classification of the $n$-hereditary quadratic monomial algebras. Let $\Lambda = kQ/I$ be such an algebra. In the first subsection, we tackle the case $n=2$. With the additional assumption that the preprojective algebra can be given by a planar selfinjective quiver with potential, we show that there are only two examples. 
Then, in the next subsection, we show that provided $n\geq 3$, the only $n$-hereditary quadratic monomial algebras are the Nakayama ones given in the previous section.


\subsection{The case $n=2$} 

The goal of this section is to prove the following theorem. 

\begin{theorem}
\label{thm:main_quad_mono}
Let $\Lambda= kQ/ I $ be a $2$-hereditary quadratic monomial algebra. Assume that $\Pi(\Lambda)$ is given by a planar quiver with potential. Then $\Lambda$ is one of the two bounded quiver algebras given in (\ref{two_quivers}). These algebras are $2$-representation-finite.
\end{theorem}

These two algebras were already known to be $2$-representation-finite. 
In fact, the first one appears already in \cite[Theorem 3.12]{IO13}. 
The second one is a cut, a notion defined below, of $\Pi(\mathbb A_3^{\operatorname{bip}}\otimes_k \mathbb A_3^{\operatorname{bip}})$, where $\mathbb A_3^{\operatorname{bip}}$ is the Dynkin quiver of type $A$ with three vertices and bipartite orientation.  

We provide more information on the preprojective algebra $\Pi(\Lambda)$ of a $2$-hereditary algebra. It is a Jacobian algebra which is selfinjective in the case when $\Lambda$ is $2$-representation-finite, and $3$-Calabi--Yau in the case when $\Lambda$ is $2$-representation-infinite. We give a brief overview of these useful facts. They are key in our classification result.

\begin{definition}
Let $Q$ be a quiver and $\mathcal J$ be the ideal generated by arrows. A \emph{potential} $W$ is an element in $\widehat{kQ}/[\widehat{kQ, kQ}]$, where $\widehat{kQ}$ is the completion of the path algebra with respect to the $\mathcal J$-adic topology.  
\end{definition}

\begin{definition}
Let $(Q, W)$ be a quiver with potential. The \emph{Jacobian algebra} of $(Q, W)$ is defined as 
\[
	\Jac(Q, W):=\widehat{kQ}/\langle\delta_a W \,| \, a \in Q_1\rangle.
\]
\end{definition}

Every $3$-preprojective algebra is a Jacobian algebra. 

\begin{theorem}\cite[Theorem 6.10]{Kel11}
\label{thm:kel}
Let $\Lambda$ be a finite-dimensional algebra of global dimension $2$. Then there exists a quiver $Q_{\Lambda}$ and a potential $W_{\Lambda}$ such that $\Pi(\Lambda)\cong \Jac(Q_{\Lambda}, W_{\Lambda})$. 
\end{theorem}

Let $M$ be a minimal set of relations in $\Lambda$. The quiver of $Q_\Lambda$ is given by adding new arrows $c_\rho:i\to j$ for every relation $\rho:j\to i$ in $M$. The potential $W_{\Lambda}$ is given by 
\[
	W_{\Lambda} = \sum_{\rho\in M} \rho c_\rho.
\]
In particular, if $\Lambda$ is quadratic, then $\Pi(\Lambda)$ is quadratic as well. 

One important assumption for the main result of this section is that $\Pi(\Lambda)$ is a \emph{planar} quiver algebra with potential. In fact, we give at the end of this subsection an example of a $2$-hereditary quadratic monomial algebra whose preprojective algebra does not satisfy this property. We provide the definition here. 

\begin{definition}
Let $Q$ be a quiver without loops or $2$-cycles. An \emph{embedding} $\epsilon: Q\to \R^2$ is a map which is injective on the vertices, sends arrows $a:i\to j$ to the open line segment $l_a$ from $\epsilon(i)$ to $\epsilon(j)$, and satisfies 
\begin{itemize}
\item $\epsilon(i)\not\in l_a$ for every $i\in Q_0$ and $a\in Q_1$ and 
\item $l_a\cap l_b = \varnothing$ for all $a\not = b\in Q_1$. 
\end{itemize} 
The pair $(Q, \epsilon)$ is called a \emph{plane quiver}. A \emph{face} of $(Q, \epsilon)$ is a bounded component of $\R^2\setminus \epsilon(Q)$ which is an open polygon. 
\end{definition}

\begin{definition}
Let $(Q, \epsilon)$ be a plane quiver such that every bounded connected component of $\R^2 \setminus \epsilon(Q)$ is a face and the arrows bounding every face are cyclically oriented. The \emph{potential induced} from $(Q, \epsilon)$ is the linear combination $W$ of the bounding cycles of all faces. The quiver with potential $(Q, W)$ is called the \emph{planar QP induced from $(Q, \epsilon)$} and any quiver with potential obtained in this way is called a \emph{planar QP}.
\end{definition}

\begin{remark}
A quiver with potential whose underlying quiver is planar is not necessarily a planar QP. In fact, the planarity has to be compatible with the potential, that is, each face is bounded by an oriented cycle. 
\end{remark}

We can obtain algebras of global dimension at most $2$ from $\Pi$ by using cuts. In fact, let $(Q, W)$ be a quiver with potential and $\CC\subset Q_1$ be a subset. We define a grading $g_{\CC}$ on $Q$ by setting 
\[
	g_{\CC}(a):= \left\{
				\begin{array}{ll}
				1 & a\in \CC\\
				0 & a\not\in\CC
				
				\end{array}
			\right.
\]
for each $a \in Q_1$.
\begin{definition}
A subset $\CC\subset Q_1$ is called a \emph{cut} if $W$ is homogeneous of degree $1$ with respect to $g_{\CC}$.
\end{definition}

When $\CC$ is a cut, there is an induced grading on $\Jac(Q, W)$. We denote by $\Jac(Q, W)_{\CC}$ the degree $0$ part with respect to this grading.   

\begin{definition}
A cut $\CC$ is called \emph{algebraic} if it satisfies the following properties: 
\begin{enumerate}
\item $\Jac(Q, W)_{\CC}$ is a finite-dimensional $k$-algebra with global dimension at most two; 
\item $\{\delta_c W\}_{c\in \CC}$ is a minimal set of generators in the ideal $\langle \delta_c W\, |\, c\in \CC\rangle$.
\end{enumerate}
\end{definition}

All truncated Jacobian algebras $\Jac(Q, W)_{\CC}$ given by algebraic cuts $\CC$ are cluster equivalent \cite[Proposition 7.6]{HI11}. These are related to $2$-APR tilts \cite{IO11}. 

When $\Lambda$ is $2$-hereditary, $\Pi$ enjoys some additional characteristics. 

\begin{proposition}
Let $\Lambda$ be a $k$-algebra such that $\gldim \Lambda\leq 2$.
\begin{itemize}
\item \cite[Theorem 5.6]{HIO14} The following are equivalent.
	\begin{enumerate}
	\item $\Pi(\Lambda) = \Jac(Q, W)$ is a bimodule $3$-Calabi--Yau Jacobian algebra of Gorenstein parameter $1$;
	\item $\Pi(\Lambda)_{\CC}$ is a $2$-representation-infinite algebra for every cut $\CC\subset Q_1$. 
	\end{enumerate}
\item \cite[Proposition 3.9]{HI11} The following are equivalent. 
	\begin{enumerate}
	\item $\Pi(\Lambda) = \Jac(Q, W)$ is a finite-dimensional selfinjective Jacobian algebra;
	\item $\Pi(\Lambda)_{\CC}$ is a $2$-representation-finite algebra for every cut $\CC\subset Q_1$. 
	\end{enumerate}
\end{itemize}
\end{proposition}

This characterisation allows us to work with the following exact sequences. 

\begin{theorem}
\label{thm:exact_si}
Let $\Pi = \Jac(Q, W)$ be a Jacobian algebra. 
\begin{itemize}
\item {\cite[Proof of Theorem 3.1]{Boc08}} $\Pi$ is $3$-Calabi--Yau if and only if the following complex of left $\Pi$-modules is exact for every simple module $S_i$: 
\begin{equation}
\label{eq:exact_cy}
	0\to P_i\xrightarrow{[a]} \bigoplus_{\substack{a\in Q_1\\ h(a) = i}} P_{t(a)}\xrightarrow{[\delta_{(a,b)}W]}\bigoplus_{\substack{b\in Q_1\\ t(b) = i}} P_{h(b)}\xrightarrow{[b]} P_i\to S_i \to 0,
\end{equation}
where $P_j:= \Pi e_j$ and $\delta_{(a,b)}W:= \delta^{\Le}_a\circ\delta^{\Ri}_bW$. 
\item {\cite[Theorem 3.7]{HI11}} $\Pi$ is selfinjective if and only if it is finite-dimensional and the following complex of left $\Pi$-modules is exact for every simple module $S_i$: 
\begin{equation}
\label{eq:exact_si}
	P_i\xrightarrow{[a]} \bigoplus_{\substack{a\in Q_1\\ h(a) = i}} P_{t(a)}\xrightarrow{[\delta_{(a,b)}W]}\bigoplus_{\substack{b\in Q_1\\ t(b) = i}} P_{h(b)}\xrightarrow{[b]} P_i\to S_i \to 0.
\end{equation}

\end{itemize}
\end{theorem}

We also need a couple of additional definitions to treat the case when $\Lambda$ is in addition a quadratic monomial algebra. Let $\Pi$ be a Jacobian algebra with potential $W$. We say that $\Pi$ admits a \emph{monomial} \emph{cut} $\CC$ if $\Pi_{\CC}$ is a monomial algebra. Also, an arrow $a$ in the quiver of $\Pi$ is called a \emph{border} if $a$ is part of exactly one summand of $W$. It is clear that if $\Pi$ is quadratic, then $W$ is a sum of cyclic paths of length $3$, since it is homogeneous of degree $1$. We call those summands \emph{triangles}. By ideas similar to Lemma \ref{lem:deltaK_n=K_{n-1}}, every arrow is part of at least one triangle.

The following lemma is clear. 

\begin{lemma}
A cut $\CC$ is monomial if and only if the arrows in degree $1$ are borders. 
\end{lemma}

In particular, the existence of a monomial cut in $\Pi$ implies that there is at least one border in each summand. An important step in our classification proof is to show that there is exactly one border, unless there is only one summand.  

The following lemma is elementary, but we include a proof for the convenience of the reader.

\begin{lemma}
\label{lem:mat_indec}
Let $\Pi = \Jac(Q,W)$ be a Jacobian algebra which is either selfinjective or Calabi--Yau. The matrix $[\delta_{(a,b)}W]$ in the complexes (\ref{eq:exact_cy}) and (\ref{eq:exact_si}) is indecomposable, that is, it is not similar to a block matrix. 
\end{lemma}

\begin{proof}
Suppose by contradiction that the complexes can be written as 
\[
	P_i\xrightarrow{\begin{bmatrix}[a'] \\
 [a'']\end{bmatrix}} \bigoplus_{\substack{a\in Q_1\\ h(a) = i}} P_{t(a)}\xrightarrow{\begin{bmatrix}
[\delta_{(a',b')}W] & [0] \\
[0] & [\delta_{(a'',b'')}W] 
\end{bmatrix}}\bigoplus_{\substack{b\in Q_1\\ t(b) = i}} P_{h(b)}\xrightarrow{\begin{bmatrix} [b'] & [b'']\end{bmatrix}} P_i\to S_i \to 0,
\]
for some vectors of arrows $[a'], [a''], [b'], [b'']$. Then the element 
$([0], [a''])\in \bigoplus_{\substack{a\in Q_1\\ h(a) = i}} P_{t(a)}$ 
is a cycle which is not a boundary, contradicting the exactness of the complex. 
\end{proof}

Using this, we now show that we can exclude an important class of examples, namely those coming from truncated Nakayama algebras.

\begin{lemma}
\label{lem:every_arrow_border}
Let $\Pi := \Jac(Q, W) $ be a Jacobian algebra which is either selfinjective or Calabi--Yau. Suppose that there exists a summand $W'$ of $W$ in which every arrow is a border. Then $\Pi$ is given by the quiver 

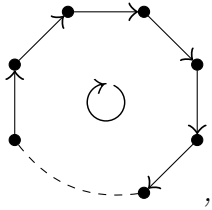
\begin{center}
\begin{tikzpicture}[auto, scale = 0.5]
	\node[circle, draw, thin,fill=black!100, scale=0.4] (0) at (1,2.4) {};
	\node[circle, draw, thin,fill=black!100, scale=0.4] (7) at (-1,2.4) {};
	\node[circle, draw, thin,fill=black!100, scale=0.4] (3) at (1,-2.4) {};
	\node[circle, draw, thin,fill=black!100, scale=0.4] (1) at (2.4,1) {};
	\node[circle, draw, thin,fill=black!100, scale=0.4] (2) at (2.4,-1) {};
	\node[circle, draw, thin,fill=black!100, scale=0.4] (6) at (-2.4,1) {};
	\node[circle, draw, thin,fill=black!100, scale=0.4] (5) at (-2.4,-1) {};
	\node[fill=none, scale=2] (8) at (0,0) {$\circlearrowright$};
	\draw[decoration={markings,mark=at position 1 with {\arrow[scale=2]{>}}},
    postaction={decorate}, shorten >=0.4pt] (0) to (1);
	\draw[decoration={markings,mark=at position 1 with {\arrow[scale=2]{>}}},
    postaction={decorate}, shorten >=0.4pt] (1) to (2);
    	\draw[decoration={markings,mark=at position 1 with {\arrow[scale=2]{>}}},
    postaction={decorate}, shorten >=0.4pt] (2) to (3);
        	\draw[decoration={markings,mark=at position 1 with {\arrow[scale=2]{>}}},
    postaction={decorate}, shorten >=0.4pt] (5) to (6);
        	\draw[decoration={markings,mark=at position 1 with {\arrow[scale=2]{>}}},
    postaction={decorate}, shorten >=0.4pt] (6) to (7);
       \draw[decoration={markings,mark=at position 1 with {\arrow[scale=2]{>}}},
    postaction={decorate}, shorten >=0.4pt] (7) to (0);
    \draw [dashed] (3) to [bend left = 30] (5);

\end{tikzpicture},
\end{center}
with potential the one obtained by summing over every cyclic rotation of the complete cycle.
\end{lemma}

\begin{proof}
Let $W' = x_n\cdots x_1$ be the summand in which every arrow is a border. By Lemma \ref{lem:mat_indec}, the matrix 
\[
	[\delta_{(a,b)}W]: \bigoplus_{\substack{a\in Q_1\\ h(a) = h(x_n)}} P_{t(a)}\rightarrow\bigoplus_{\substack{b\in Q_1\\ t(b) = h(x_n)}} P_{t(b)},
\]
is indecomposable. However, the column 
\[
	[\delta_{(a,x_1)}W]_{h(a) = h(x_n)}
\]
and the row 
\[
	[\delta_{(x_n,b)}W]_{t(b) = h(x_n)}
\]
each only have one non-zero element, since $x_1$ and $x_n$ are borders. Thus, $[\delta_{(a,b)}W] $ can only be indecomposable if its dimension is $1\times 1$. This means that $x_n$ is the only arrow ending at $h(x_n)$ and $x_1$ is the only arrow starting at $h(x_n)$. Repeating the argument with $x_1$, $x_2$, $\ldots$, $x_{n-1}$, we deduce that there is also only one arrow ending and one starting at $h(x_i)$, for $i =1,\ldots, n-1$, as well. Since the quiver $Q$ is connected, $\Pi$ must be given by the QP described in the statement. 
\end{proof}

Now assume that $\Pi$ is the preprojective algebra of a $2$-hereditary algebra. If $\Pi$ admits a monomial cut, then we show that summands of the potential cannot have two borders either. For this, we need the following proposition. It is shown in \cite{HI11} in the case where $\Jac(Q, W)$ is a selfinjective algebra, but the same proof also works in the case when $\Jac(Q, W)$ is a $3$-Calabi--Yau algebra. 

\begin{proposition}[{\cite[Proposition 3.10]{HI11}}]
\label{prop:cut_alg}
Let $\Jac(Q, W)$ be a preprojective algebra over a $2$-hereditary algebra. Then every cut $\CC\subset Q_1$ is algebraic. 
\end{proposition}

\begin{lemma}
\label{lem:arrow_two_relations}
Let $\Pi = \Jac(Q, W)$ be a quadratic Jacobian algebra which is either selfinjective or $3$-Calabi--Yau. Suppose that $\Pi$ admits a monomial cut. Then there does not exist a triangle $W'$ of $W$ which has exactly two borders. 
\end{lemma}

\begin{proof}
Suppose by contradiction that $W' = xyz$ is a triangle of $W$ such that $x$ and $y$ are both borders, and $z$ is not. Then there is another summand $W'' = uvz$ containing $z$. Let $\CC$ be the monomial cut on $\Pi$, in which we may assume without loss of generality that $x$ is in degree $1$. Then, the grading $\CC'$ obtained from $\CC$ by putting $x$ in degree $0$ and $y$ in degree $1$ is also a cut, since $x$ and $y$ are borders which do not appear in other triangles. Now, suppose that $v$ is in degree $1$. Then, in $\Pi_\CC$, we have that $zu = 0$ and $yz=0$. Thus $\gldim \Pi_\CC\geq 3$, and $\CC$ is not algebraic, contradicting Proposition \ref{prop:cut_alg}. Similarly, if $u$ is in degree $1$, then $\CC'$ is a non algebraic cut. As $z$ cannot be in degree $1$ in a monomial cut, $W''$ cannot be put in degree $1$ in $\CC$.
\end{proof}

Combining the previous two lemmas, we obtain the following corollary. 

\begin{corollary}
Let $\Pi = \Jac(Q, W) $ be a quadratic Jacobian algebra which is selfinjective or $3$-Calabi--Yau and admits a monomial cut. Then either every summand of $W$ has exactly one border, or $\Pi$ is the quiver algebra with potential with a unique triangle: 
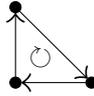
\begin{center}
\begin{tikzpicture}[auto]
	\node[circle, draw, thin,fill=black!100, scale=0.4] (0) at (0,0) {};
	\node[circle, draw, thin,fill=black!100, scale=0.4] (1) at (0,1) {};
	\node[circle, draw, thin,fill=black!100, scale=0.4] (2) at (1,0) {};
	\node[fill=none, scale=1] (3) at (0.33,0.33) {$\circlearrowright$};
	\draw[decoration={markings,mark=at position 1 with {\arrow[scale=2]{>}}},
    postaction={decorate}, shorten >=0.4pt] (0) to (1);
	\draw[decoration={markings,mark=at position 1 with {\arrow[scale=2]{>}}},
    postaction={decorate}, shorten >=0.4pt] (1) to (2);
    	\draw[decoration={markings,mark=at position 1 with {\arrow[scale=2]{>}}},
    postaction={decorate}, shorten >=0.4pt] (2) to (0);
\end{tikzpicture}
\end{center}
\end{corollary}
Note that the latter case is the preprojective algebra of $k\mathbb A_3/\mathcal J^2$, the first example in our main theorem. 

The vanishing-of-Ext condition also gives information about the quiver of $2$-hereditary quadratic monomial algebras. We have the following corollary to Proposition \ref{prop:arrow_start_end}.

\begin{corollary}
\label{cor:rel_sink_source}
 Let $\Lambda$ be a quadratic monomial algebra with $\Ext^1_{\Lambda^e}(\Lambda, \Lambda^e) = 0$. Then for every relation $\rho = ba$, the vertex $h(b)$ is a sink and the vertex $t(a)$ is a source. 
 \end{corollary} 

\begin{proof}
Since $\gldim\Lambda = 2$, every arrow is either the start or the end of every relation they divide. The result thus follows directly from Proposition \ref{prop:arrow_start_end}. 
\end{proof}

This leads to the following definition.

\begin{definition}
\label{def:star}
Let $r,s\in\Z_{\geq 1}$. The \emph{$(r,s)$-star quiver}, denoted by $S_{(r,s)}$, is the quiver
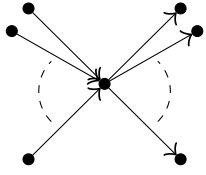
\begin{center}
\begin{tikzpicture}[auto]
	\node[circle, draw, thin,fill=black!100, scale=0.4] (1) at (2,0) {};
	\node[circle, draw, thin,fill=black!100, scale=0.4] (2) at (1,1) {};
	\node[circle, draw, thin,fill=black!100, scale=0.4] (4) at (0,2) {};
	\node[circle, draw, thin,fill=black!100, scale=0.4] (5) at (2,2) {};
	\node[circle, draw, thin,fill=black!100, scale=0.4] (6) at (0,0) {};
	\node[circle, draw, thin,fill=black!100, scale=0.4] (7) at (-0.22,1.7) {};
	\node[circle, draw, thin,fill=black!100, scale=0.4] (8) at (2.22,1.7) {};
	\draw[decoration={markings,mark=at position 1 with {\arrow[scale=2]{>}}},
    postaction={decorate}, shorten >=0.4pt] (2) to (1);
 	\draw[decoration={markings,mark=at position 1 with {\arrow[scale=2]{>}}},
    postaction={decorate}, shorten >=0.4pt] (4) to (2);
	\draw[decoration={markings,mark=at position 1 with {\arrow[scale=2]{>}}},
    postaction={decorate}, shorten >=0.4pt] (2) to (5);
    	\draw[decoration={markings,mark=at position 1 with {\arrow[scale=2]{>}}},
    postaction={decorate}, shorten >=0.4pt] (6) to (2);
        	\draw[decoration={markings,mark=at position 1 with {\arrow[scale=2]{>}}},
    postaction={decorate}, shorten >=0.4pt] (7) to (2);
            	\draw[decoration={markings,mark=at position 1 with {\arrow[scale=2]{>}}},
    postaction={decorate}, shorten >=0.4pt] (2) to (8);
    \draw[dashed] (0.3,0.5) to [bend left=45] (0.3, 1.3) {};
    \draw[dashed] (1.7,0.5) to [bend right=45] (1.7, 1.3) {};
\end{tikzpicture}
\end{center}
with $r+s+1$ vertices and a central vertex $z$ which is the head of $r$ arrows and the tail of $s$ arrows. We always denote the arrows $i\to z$ by $a_i$ and the arrows $z\to j$ by $b_j$. 
\end{definition}

We conclude that every quadratic monomial $2$-hereditary algebra is a bound quiver algebra over a star quiver. 

\begin{corollary}
\label{cor:star}
Let $\Lambda$ be a quadratic monomial algebra with $\Ext^1_{\Lambda^e}(\Lambda, \Lambda^e) = 0$. Then the quiver of $\Lambda$ is an $(r,s)$-star quiver. In particular, $2$-hereditary quadratic monomial algebras are given by quotients of $(r,s)$-star quiver algebras.
\end{corollary}

\begin{proof}
By Corollary \ref{cor:rel_sink_source}, every relation is a path of length $2$ which starts at a source and ends at a sink. Furthermore, Proposition \ref{prop:nosinksource} implies that these vertices can only be source and sink to one arrow. This means that the quiver of $\Lambda$ is made of paths of length $2$ which all intersect at a common middle vertex.
\end{proof}

Before completing the proof of the main theorem of this section, we explore further some quick restrictions on the relations which are imposed by the vanishing-of-Ext condition. From now on in this section, we let $\Lambda$ be a bound $(r,s)$-star quiver algebra. For each arrow $a$ such that $h(a) = z$, we define 
\[
	\mathcal Z_a := \{b:z\to j\, |\, ba = 0\}
\]
and we define a set $\mathcal Z_b$ similarly for arrows $b$ such that $t(b) = z$. 
By Lemma \ref{lem:arrow_two_relations}, we have that $|\mathcal Z_a|$ and $|\mathcal Z_b|$ are greater than or equal to $2$, unless $(r,s) = (1,1)$. 

\begin{lemma}
\label{lem:no_sub_rel}
Let $\Lambda$ be as above. If there are two distinct arrows $a$ and $a'$ such that $\mathcal Z_a \subset \mathcal Z_{a'}$, then $\Ext^1_{\Lambda^e}(\Lambda, \Lambda^e) \not= 0$. 
\end{lemma}

\begin{proof}
Suppose that the two arrows in the statement are such that $h(a) = h(a') = z$, the case where $t(a) = t(a') = z$ being similar. 
Consider then the element 
\[
	(e_{t(a')}\otimesk e_z a)_{a'}\in \Lambda e_{t(a')}\otimesk e_z \Lambda.
\]
This is a cocycle since $\mathcal Z_a\subset \mathcal Z_{a'}$. It is however not a coboundary, by the same principles as in section \ref{sec:trunc}. 
\end{proof}

We obtain the following corollary as a particular case.

\begin{corollary}
\label{cor:not_all_rel}
Let $\Lambda$ be as above and assume that $\Ext^1_{\Lambda^e}(\Lambda, \Lambda^e) = 0$. Suppose that $s\geq 2$ and let $a$ be an arrow such that $h(a) =z$. Then $|\mathcal Z_a|\leq s-1$. Similarly, if $r\geq 2$ and $b$ is an arrow such that $t(b) = z$, then $|\mathcal Z_b| \leq r-1$.
\end{corollary}

We now show that the upper bound on $|\mathcal Z_a|$ is even smaller if $\Lambda$ is $2$-RF.

\begin{lemma}
\label{lem:no_m-1_rel}
Let $\Lambda$ be as above and suppose that $\Lambda$ is $2$-RF. Suppose that $s\geq 2$ and let $a$ be an arrow such that $h(a) =z$. Then $|\mathcal Z_a|\leq s-2$. Similarly, if $r\geq 2$ and $b$ is an arrow such that $t(b) = z$, then $|\mathcal Z_b| \leq r-2$.
\end{lemma}

\begin{proof}
We can use a Loewy length ($\ell\ell$) argument as follows. Suppose that there is an arrow $a_i:i\to z$ such that $|\mathcal Z_{a_i}| = s-1$. Consider the preprojective algebra $\Pi$ over $\Lambda$. We refer to its description below Theorem \ref{thm:kel}. 
Let $P_i := \Pi e_i$. We show that $\ell\ell(P_i) = 3$, whereas $\ell\ell (P_z)\geq 4$. 
Since $\Pi$ is selfinjective, this contradicts \cite[Theorem 3.3]{MV99}. 

Let $b_j:z\to j$ be the only arrow such that $b_ja_i\not =0$. Let $\rho = b_ja_{x}$ be a relation in $\Lambda$ and $c_\rho$ be the corresponding arrow in $\Pi$. Then $c_\rho b_ja_i = -\sum c_{\rho'} b_{j'}a_i = 0$, where the sum is taken over relations of the form $\rho':=b_{j'}a_x$ in $\Lambda$ which are not equal to $\rho$. The sum is not empty since $s\geq 2$. 
This shows that $\ell\ell(P_i) = 3$. Now, since a path of the form $a_\nu c_\rho b_\mu$ is never $0$ in $\Pi$ for any vertices $\nu, \mu$ and relations $\rho$, we have $\ell\ell(P_z)\geq 4$. 
Here, we have used the fact that $|\mathcal Z_{b_\mu}|,|\mathcal Z_{a_\nu}|\geq 2$. The argument is dual for an arrow $b:z\to i$ such that $|\mathcal Z_b| = r-1$.
\end{proof}

In particular, this implies that, if $\Lambda$ is $2$-RF, then either $(r,s) = (1,1)$, or $r, s \geq 4$.

We now have plenty of tools to give a full classification of the monomial $2$-hereditary algebras whose preprojective algebras is a planar quiver with potential. We prove the main theorem of this section. Note that, for the previous results of this section, we have not assumed that the preprojective algebra is a planar QP. We need the hypothesis now.

\begin{proof}[Proof of Theorem \ref{thm:main_quad_mono}]
By reasons given above, one can easily check that the two bound quiver algebras described in (\ref{two_quivers}) are $2$-RF. 
Assume that $\Lambda$ is a $2$-hereditary quadratic monomial algebra whose preprojective algebra is a planar quiver with potential. 
We prove that they are the only ones coming from a planar quiver with potential. 

By Corollary \ref{cor:star}, $\Lambda$ is an $(r,s)$-star quiver. By \cite[Proposition 3.15]{Pet19}, the planarity assumption allows us to conclude that every arrow in $\Pi(\Lambda)$ is contained in at most two summands of the potential $W$. Combining this with Lemma \ref{lem:arrow_two_relations}, we see that every arrow in $\Lambda$ is part of exactly $2$ relations. Therefore, the quiver of $\Pi(\Lambda)$ is given by the intersection of oriented triangles which all share a common vertex $z$, thus forming a regular polygon shape. In particular, $r=s$. In addition, if $\Lambda$ is $2$-RF, then we have that $r,s\geq 4$ by Lemma \ref{lem:no_m-1_rel}, unless $(r,s) = (1,1)$. If $\Lambda$ is $2$-RI, then we also obtain the same conclusion, since in the case $r=2$ or $r=3$, the preprojective algebra is clearly finite-dimensional. If $r=1$ or $r=4$, then we recover the bound quiver algebras described in (\ref{two_quivers}).  

Assume that $r\geq 5$. We show that $\Ext^1_{\Lambda^e}(\Lambda, \Lambda^e)\cong\Ext^1_{\Lambda}(D\Lambda, \Lambda) \not = 0$. 
Let $I_m$ be the injective module associated to a sink $m$ and $b: z\to m$. Also recall that $\mathcal Z_b^\complement := Q_1\setminus{\mathcal Z_b}$. 
Then $|\mathcal Z_b^\complement| = r-2$. 
Let $a_1,\ldots, a_{r-2}$ be the arrows in $\mathcal Z_b^\complement$ and define $t_i:= t(a_i)$ for $i = 1, \ldots, r-2$. 
Without loss of generality, we can assume that we ordered the arrows so that $|\mathcal Z_{a_i}\cap \mathcal Z_{a_{i+1}}| = 1$ for $i = 1, \ldots, r-3$. 
This is due to the planarity assumption on $\Pi(\Lambda)$. 
We call $b_i$ the arrow in this intersection for $i = 1, \ldots, r-3$ and define $h_i:= h(b_i)$. 
Then the projective resolution of $I_m$ is given by 
\[
	0\to \bigoplus_{i = 1, \ldots, r-3}P_{h_i}\to P_z^{r-3}\to \bigoplus_{i = 1, \ldots, r-2}P_{t_i}\to 0. 
\]
Applying $\Hom_{\Lambda}(-, \Lambda e_{t_2})$, we obtain a complex 
\[
	0\to \bigoplus_{i = 1, \ldots, r-2}e_{t_i}\Lambda e_{t_2}\to (e_z\Lambda e_{t_2})^{r-3}\to \bigoplus_{i = 1, \ldots, r-3}e_{h_i}\Lambda e_{t_2} \to \Ext^2_{\Lambda}(I_m, \Lambda e_{t_2})\to 0. 
\]
This complex is not exact at $(e_z\Lambda e_{t_2})^{r-3}$  since $\dim_k  (\bigoplus_{i = 1, \ldots, r-2}e_{t_i}\Lambda e_{t_2}) = 1$, $\dim_k ((e_z\Lambda e_{t_2})^{r-3}) = r-3$ and $\dim_k (\bigoplus_{i = 1, \ldots, r-3}e_{h_i}\Lambda e_{t_2}) = r-5$. The last equality can be explained by the fact that $b a_{2} = 0$ for $b \in Q_1$ if and only if $b = b_1$ or $b_2$. 

Thus $\Ext^1_{\Lambda}(D\Lambda, \Lambda) \not = 0$. Note that we could have chosen to take $\Hom_{\Lambda}(-, \Lambda e_{t_\mu})$ for any $\mu = 2,\ldots, r-3$ and still obtain the same conclusion. 
\end{proof}


\begin{example}
\label{ex:2rf}
We give an example of a quadratic monomial 2-RF algebra whose $3$-preprojective algebra is a non-planar selfinjective quiver with potential. If $(r,s) = (9,6)$ with arrows $a_i \colon i \rightarrow z$ for $i = 1,\ldots,9$ and $b_j \colon z \rightarrow j$, for $j = 1,\ldots,6$ one gets an example with relations 
\[
b_1 a_1, b_4 a_1, 
b_1 a_2, b_5 a_2,
b_1 a_3, b_6 a_3,
b_2 a_4, b_4 a_4,
b_2 a_5, b_5 a_5,
b_2 a_6, b_6 a_6,
b_3 a_7, b_4 a_7,
b_3 a_8, b_5 a_8,
b_3 a_9, b_6 a_9.
\]

This can be seen to be a cut of $\Pi(\mathbb{D}_4 \otimes \mathbb{D}_4)$, where both copies of $\mathbb{D}_4$ are oriented with arrows going out of the central vertex. 
Since $\mathbb{D}_4$ with this orientation is $\ell$-homogeneous, \cite[Proposition 1.4]{HI11b} implies that the tensor product is $2$-RF, and hence, this example is also $2$-RF. As the quiver of $\Pi(\mathbb{D}_4 \otimes \mathbb{D}_4)$ is a non-planar graph, the example is non-planar, too.

We note that by observing that to get a quadratic monomial cut the QP must have ``enough'' borders, it is not too hard to see that the only possible tensor products of Dynkin diagrams that have $3$-preprojective algebras with such cuts involve $\mathbb{A}_3$ and $\mathbb{D}_4$ with bipartite orientation. Moreover, $\mathbb{A}_3 \otimes \mathbb{D}_4$ can be checked to not be $2$-RF, and $\mathbb{A}_3 \otimes \mathbb{A}_3$ yields the planar example.
\end{example} 

\begin{remark} 
One should note that many natural constructions on algebras that preserve the property of being $n$-hereditary do not necessarily preserve being monomial. For instance, this includes tensor products and certain skew-group ring constructions.

Moreover, while there exist other non-planar examples, the ones we know of are all fairly large and are somewhat more complicated than the one mentioned above. 
\end{remark}


\subsection{The case $n\geq 3$} 

We classify all $n$-representation-finite quadratic monomial algebras of global dimension higher than $2$. Note that we do not assume that the preprojective algebra is a planar QP.  

\begin{theorem}
\label{thm:mono_higher}
With the exception of $k\mathbb A_{n+1}/\mathcal J^2$, there are no quadratic monomial $n$-RF algebras for $n\geq 3$. 
\end{theorem}

\begin{proof}
To begin with, we observe that, by Proposition \ref{prop:arrow_start_end}, every arrow in $\Lambda$ lies on some maximal path 
\[
0 \rightarrow 1 \rightarrow 2 \rightarrow \cdots  \rightarrow i \rightarrow i + 1 \rightarrow \cdots \rightarrow n - 2 \rightarrow n - 1 \rightarrow n
\]
in which every two consecutive arrows are a relation. Also note that $0$ must be a source and $n$ a sink. 

We begin by showing that there cannot exist an arrow in $\Lambda$ different from the one in the diagram above leaving a vertex $i$ with $i < n-1$. Indeed, say there was some such arrow $a_i' \colon i \rightarrow j$. Since $\Pi(\Lambda)$ is selfinjective, the projective at $i$ over $\Pi(\Lambda)$ cannot have a non-simple socle. Hence, there must be a commutation relation in $\Pi(\Lambda)$ starting at $i$ of the form $\alpha_ira_i + \Sigma_{k} \alpha_kr_k a_{i,k}$ with arrows $r$ and $r_k$ in $\Pi(\Lambda)$ (but not in $\Lambda$) and with $\alpha_j$ scalars, and $\alpha_i, \alpha_j \neq 0$ for some $j$. 

Indeed, to see that the latter claim must hold, note that if $pa_i,qa_i' \in \socmod \Pi(\Lambda)e_i$, then $pa_i - qa_i' \in I = \langle \rho_j \rangle$ with $\{\rho_j\}$ a set of minimal relations. In other words, $pa_i - qa_i' = \Sigma_j u_j \rho_j v_j$ where $u_j,v_j$ can be assumed to be paths up to scalars. 
Rewriting this as
\begin{align*}
pa_i - qa_i' 
& = \Sigma_j u_j \rho_j v_j \\ 
& = \Sigma_k u_k \rho_k v_k + \Sigma_l u_l \rho_l a_i + \Sigma_m u_m \rho_m a_i'
\end{align*}
with $v_k \neq a_i, a_i'$, we see thus that if $pa_i,qa_i'$ are non-zero in $\Pi(\Lambda)$ and $a_i\neq a_i'$, then
\[
 \Sigma_k u_k \rho_k v_k = pa_i - qa_i' - \Sigma_l u_l \rho_l a_i - \Sigma_m u_m \rho_m a_i' \neq 0.
\]
This also implies that we can set $v_k = 1$. 
Moreover, we observe that some $\rho_{j}$ is of the form $\alpha_{i} r a_{i} + \Sigma_{k} \alpha_{k} r_{k} a_{i,k}$ as stated above since $pa_{i}$ must occur as some term in some $u_{k} \rho_k$. 
If $pa_{i}$ was the only non-zero term in that $u_{k} \rho_{k}$, we would have $pa_{i} = 0$ in $\Pi(\Lambda)$, contrary to our assumptions. This establishes the claim.

Note that $\Lambda$ is Koszul, so by Proposition \ref{thm:quiver_construction}, we know that such a commutation relation and new arrows beginning in vertices $i+1$ and $j$ in the preprojective correspond exactly to elements in $K^n$ ending with arrows $a_i \colon i \rightarrow i +1$ and $a_i' \colon i \rightarrow j$, and differing only in the final arrow. 
However, there can be no such element ending in $i+1$ as $i+1 < n$ is not a sink. 

Since $\Lambda$ is $n$-RF if and only if $\Lambda^{\op}$ is $n$-RF, we have also shown that there are no arrows ending in $i$ with $1 < i$. Hence, without loss of generality, we can assume that if $\Lambda$ has quiver different from linearly oriented $\mathbb{A}_n$, then there must be an arrow $a_{n-1}'$ different from $a_{n-1}$ starting in $n-1$. 

Yet, if this was the case, the $\Pi(\Lambda)$-projective at $n-1$ would be of Loewy length $\geq 2$ whereas the $\Pi(\Lambda)$-projective at $i < n - 2$ would be of Loewy length $\leq 1$, as there cannot be any new arrows in $\Pi(\Lambda)$ not in $\Lambda$ out of $n-1$ as it is not a sink. This yields a contradiction by the fact that $\Pi(\Lambda)$ has homogeneous relations and \cite[Theorem 3.3]{MV99}. By using Vaso's classification (\cite{Vas19}) of $n$-RF algebras that are quotients of Nakayama algebras, we are done.
\end{proof}


\bibliographystyle{alpha}

\bibliography{bib_quiver_properties_hereditary.bib}


\end{document}